\documentclass[11pt]{article}

\usepackage{setspace}

\usepackage{amsmath}
\usepackage{amssymb}
\usepackage{amsthm}
\usepackage{titlesec}
\usepackage{graphicx}
\usepackage{caption}
\usepackage{subcaption}

\usepackage{diagbox}

\usepackage[T1]{fontenc}
\usepackage[utf8]{inputenc}

\usepackage{comment}

\usepackage{indentfirst}
\usepackage[top=1.25in,left=1.25in,right=1.25in]{geometry}
\newlength{\numone}
\newlength{\widone}
\newlength{\numtwo}
\newlength{\widtwo}
\newcounter{countp}

\newtheorem{thm}{Theorem}

\newtheorem{lemma}[thm]{Lemma}
\newtheorem{prop}[thm]{Proposition}

\theoremstyle{definition}
\newtheorem{defin}[thm]{Definition}

\numberwithin{equation}{section}

\author{\Large{Riccardo W. Maffucci}}
\newcommand{\Addresses}{{
		\footnotesize
		
		R.W.~Maffucci, \textsc{EPFL MA SB,
			Lausanne, Switzerland 1015}\par\nopagebreak\vspace{-0.35cm}
		\textit{E-mail address}, R.W.~Maffucci: \texttt{riccardo.maffucci@epfl.ch}}}

\title{\Large{\uppercase{\bf Characterising $3$-polytopes of radius one with unique realisation}}}

\date{}

\def\calB{\mathcal{B}}
\def\calC{\mathcal{C}}

\newcommand{\C}{\mathbb{C}}

\begin{document}
\titleformat{\section}
  {\Large\scshape}{\thesection}{1em}{}
\titleformat{\subsection}
  {\large\scshape}{\thesubsection}{1em}{}
\maketitle


\begin{abstract}
Let $F$ be a planar, $3$-connected graph of radius one on $p$ vertices, with $a$ vertices of degree three. We characterise all unigraphic degree sequences for such graphs, when $a\geq 3$ and $p$ is large enough with respect to $a$. This complements the work of \cite{mafpo4} for $a\leq 3$.
\end{abstract}

\begingroup
\setstretch{-0.5}
\tableofcontents
\endgroup

\section{Introduction}
\subsection{Main Theorem}
The degree sequence of a graph $F$ is
\begin{equation}
\label{eqn:s}
s: d_1,d_2,\dots,d_p
\end{equation}
where $p=|V(F)|$ and $d_i=\deg(v_i)$ for $i=1,2,\dots,p$. We say that a graph has unique realisation, or is unigraphic, if up to isomorphism it is the only graph with such a sequence. For work on unigraphic sequences, see e.g. \cite{koren1,li1975,john80}. A graph has radius one if there is a vertex adjacent to all others. A $3$-polytope is a planar, $3$-connected graph \cite{radste}.

In \cite{mafpo4}, we considered the problem of finding the unigraphic $3$-polytopes of radius one, and we completely characterised these in the cases $a=2$ and $a=3$, where $a$ is the number of vertices of $F$ of degree $3$ (i.e. the number of $3$'s that appear in \eqref{eqn:s}). It is not difficult to show that $a\geq 2$ \cite[Lemma 5]{mafpo4}.

In this paper, we establish a more general result for any feasible value of $a\geq 3$. This complements the work of \cite{mafpo4} for $a=2$. Our main result characterises the families of unigraphic sequences where $a\geq 3$ and $p$ is large enough compared to $a$.
\begin{thm}
	\label{thm:1}
	Let $s$ be a sequence as in \eqref{eqn:s}, and denote by $a$ the number of degree three vertices. Assuming $a\geq 3$ and $p\geq 3a$, then $s$ is unigraphic as a $3$-polytope of radius $1$ if and only if $s$ is one of the following:
	\begin{align*}
	\\B1&: p-1, (x+3)^2, p-1-2x+3, 4^{p-7}, 3^3, & & p\geq 10, \ 3\leq x\leq\lfloor(p-4)/2\rfloor;
	\\B2&: p-1, \left(\frac{p+1}{4}+3\right)^4, 4^{p-9}, 3^4,& & p\geq 15, \ p\equiv 3\hspace{-0.25cm}\pmod 4;
	\\B3&: p-1, \left(\frac{p+3}{5}+3\right)^5, 4^{p-11}, 3^5,& & p\geq 22, \ p\equiv 2\hspace{-0.25cm}\pmod 5;
	\\C&: p-1, x+3, 2(a-1)-x+3, 5^{p-a-3}, 3^a,& & \substack{p\geq 8, \ 3\leq a\leq p/3, \\\ 1+2\lceil(a-2)/2\rceil\leq x\leq 2a-3, \ x \text{ odd};}
	\\D&: p-1, \left(\frac{p}{4}+3\right)^4, 4^{p-8}, 3^3, & & p\geq 12, \ p\equiv 0\hspace{-0.25cm}\pmod 4,
	\end{align*}
	or the exceptional $14,5^9,3^5$.
\end{thm}

We point out that once $a\geq 3$ is fixed and $p$ grows, \textit{all except finitely many} unigraphic $3$-polytopes of radius $1$ with $a$ vertices of degree three will be of one of the types B1, B2, B3, C, D.

On the other hand, if we fix $p$, it is straightforward to see that, if $s$ is unigraphic and $F$ is not a pyramid (i.e. wheel graph), then $2\leq a\leq p/2$ (for the upper bound e.g. one may adapt \cite[Lemma 8]{mafpo4}). Then Theorem \ref{thm:1} classifies the unigraphic $s$ for $2\leq a\leq p/3$ but not for $p/3<a\leq p/2$.

Note that \cite[Theorem 3]{mafpo4} is the special case $a=3$ of Theorem \ref{thm:1}. In \cite{mafpo4}, this special case was proven in a different, more straightforward way. The sequence type C1 from \cite[Theorem 3]{mafpo4} is just the special case $a=3$ of type C from Theorem \ref{thm:1}, and the exceptional $5,4^2,3^3$ and $6,5^3,3^3$ from \cite[Theorem 3]{mafpo4} both satisfy $p\leq 3a-1$.

The unigraphic $3$-polytopes corresponding to these sequences will be described once we have recalled some notation from \cite{mafpo4}.

\begin{defin}
	\label{def:G}
For $F$ a $3$-polytope satisfying $\deg(v_1)=p-1$,  $a\geq 3$, and $p\geq 3a$, we note that $F-v_1$ is a non-empty, planar, Hamiltonian graph. It has a region containing all of its vertices. Henceforth, $H$ will denote a Hamiltonian cycle in $F-v_1$, and 
\[G:=F-v_1-E(H),\]
a non-empty planar graph, of sequence
\begin{equation}
\label{eqn:s'}
s': d_2-3, d_3-3, \dots, d_{p}-3.
\end{equation}
Exactly $a$ entries in $s'$ are zeroes. The isolated vertices of $G$ form the set $Z$. 
\end{defin}

Now we are in a position to better describe the polytopes $F$ corresponding to the sequences in Theorem \ref{thm:1}. We may do this by characterising $G-Z$. In the exceptional $14,5^9,3^5$, $G-Z$ is the disjoint union of three triangles. In all remaining cases $G-Z$ is connected. For types B1, B2, and B3, $G-Z$ is a triangle, triangulated quadrilateral (i.e., diamond graph) and triangulated pentagon respectively, together with vertices of degree one adjacent to the boundary points of the triangle, or triangulated quadrilateral or pentagon. In B1, at least two of the vertices of degree $>1$ have the same degree; in B2 (resp. B3), all of the $4$ (resp. $5$) vertices of degree $>1$ have the same degree.

For C, $G-Z$ is formed of a set of triangles sharing a vertex $u$, another (possibly empty) set of triangles sharing a vertex $v$, and a $uv$-path (that may be trivial, i.e. possibly $u=v$). 
For D, $G-Z$ is a triangle together with a fourth vertex adjacent to exactly one point on the boundary of the triangle, and with extra degree one vertices adjacent each to one of these four vertices, so that these four have the same degree in $G$.

\subsection{Overview of the proof}
The rest of this paper is dedicated to proving Theorem \ref{thm:1}. In section \ref{sec:g41}, we will prove that apart from the exceptional case $14,5^9,3^5$, if $p\geq 3a$ then $G-Z$ is connected. In the second part of the proof (section \ref{sec:g42}), we will analyse several cases for connected $G-Z$, and assuming $s$ is unigraphic, either determine $F$ or bound $p$ with respect to $a$. The first case is when $G-Z-Y$ is $2$-connected, where $Y$ is the set of degree $1$ vertices in $G$ (section \ref{sec:2con}). The second and third cases are for $G-Z-Y$ not $2$-connected, distinguishing between when no block of $G-Z-Y$ contains all vertices that are separating in $G$ (section \ref{sec:c2}), and when such a block exists (section \ref{sec:c3}). A proposition summarising results closes each section. Combining these propositions, we will prove Theorem \ref{thm:1} (section \ref{sec:end}).

One general idea is that, outside of the types of sequence listed in Theorem \ref{thm:1}, the number of vertices of $G$ not lying on a cyclic block of $G$ is \textit{bounded}. 
This fact 
allows us to obtain an upper bound for $p$ depending on $a$.

Everywhere $V$ and $E$ indicate vertex and edge sets. A \textit{caterpillar} is a tree graph where every vertex is within distance one of a central path $c_1,c_2,\dots,c_\ell$, $\ell\geq 1$. We will denote catepillars by $\calC(x_1,\dots,x_\ell)$, where $x_i=\deg(c_i)$.  
A cyclic graph is a graph containing a cycle. A cyclic block is thus any block other than $K_2$. A block of a graph $G$ with only one vertex separating in $G$ is called an endblock.



\subsection{Acknowledgements}
The author was supported by Swiss National Science Foundation project 200021\_184927 held by Prof. M. Viazovska.

\subsection{Data availability statement}
All data generated during this study are included in this article.


\section{First part of the proof}
\label{sec:g41}
Henceforth we assume that the number of degree three vertices in $F$ is $a\geq 3$. Recall Definition \ref{def:G} for $H,G,s',Z$.

Firtly, it is straightforward to adapt the proof of \cite[Lemma 8]{mafpo4} to see that, when $a\geq 3$, $G$ has at least one cycle, with the only exception of $s: 5,4^2,3^3$ (where $p=6$). 
Next, we consider the number $k$ of cyclic connected components in $G$.

\begin{lemma}
	\label{le:pre}
If $s$ is unigraphic, $p\geq 7$, and $a\geq 3$, then $G$ has one, two or three cyclic connected components, and if three, then they are all triangles.
\end{lemma}
\begin{proof}
Denote by $G_l$, $1\leq l\leq k$, the cyclic connected components of $G$. For fixed $l$, denote by $u_{l,j}$, $1\leq j\leq i_l$, the $i_l\geq 3$ vertices of $G_l$ in the order that they appear around the Hamiltonian cycle $H$, clockwise starting from $u_{1,1}$. They may be ordered around $H$ as
\[u_{1,1},u_{2,1},\dots,u_{k,1},u_{k,2},\dots,u_{k,i_k},u_{k-1,2},\dots,u_{k-1,i_{k-1}},\dots,u_{2,2},\dots,u_{2,i_2},u_{1,2},\dots,u_{1,i_1}\]
or as
\[u_{1,1},u_{2,1},\dots,u_{2,i_2},u_{1,2},u_{3,1}\dots,u_{3,i_3},\dots,u_{1,i_{1}},u_{i_{1}+1,1},\dots,u_{i_{1}+1,i_{i_{1}+1}},\dots,u_{k,1}\dots,u_{k,i_k}\]
where in the second ordering $u_{l,j}$ actually appears only for $l\leq k$. In the first ordering we have four consecutive vertices belonging to four different components, namely $u_{1,1},u_{2,1},u_{3,1},u_{4,1}$, \textit{unless} $\mathit{k\leq 3}$, whereas it is easy to see that this cannot happen in the second ordering (recall that the $G_l$ are cyclic, hence $i_l\geq 3$ for every $1\leq l \leq k$). Thereby, if $s$ is unigraphic, then necessarily $k\leq 3$.

Now suppose that $k=3$. The second of the above two orderings reads
\begin{equation}
\label{eqn:o1}
u_{1,1},u_{2,1},\dots,u_{2,i_2},u_{1,2},u_{3,1}\dots,u_{3,i_3},u_{1,3},u_{1,4},\dots,u_{1,i_1}.
\end{equation}
We also have the feasible
\begin{equation}
\label{eqn:o2}
u_{1,1},u_{2,1},\dots,u_{2,i_2},u_{1,2},u_{1,3},u_{3,1}\dots,u_{3,i_3},u_{1,4},\dots,u_{1,i_1}.
\end{equation}
In \eqref{eqn:o1}, the vertex $u_{1,2}$ from $G_1$ does not follow or precede any other vertex from $G_1$. In \eqref{eqn:o2}, there is no vertex from any of $G_1,G_2,G_3$ that does not follow or precede any other vertex from the same component, \textit{unless} $\mathit{i_1=3}$ (where $u_{1,1}$ indeed has such property). Therefore, if $s$ is unigraphic and $k=3$, then $G_1$ has exactly three vertices and is cyclic, i.e. it is a triangle. We may change the roles of $G_1,G_2,G_3$ in \eqref{eqn:o1} and \eqref{eqn:o2} to show that $G_2,G_3$ are triangles as well.
\end{proof}

Next, we show that under the same assumptions, there can be no non-trivial tree components.
\begin{lemma}
	\label{le:notree}
	If $s$ is unigraphic, $p\geq 7$, and $a\geq 3$, then $G$ has no non-trivial tree components.
\end{lemma}
\begin{proof}
By Lemma \ref{le:pre}, there is at least one cyclic component of $G$. If there are two or more, and if by contradiction there is at least one non-trivial tree component, then there are three non-trivial components that are not all isomorphic. A slight generalisation of the scenario in $a=2$ of two copies of $K_2$ and a star that is not $K_2$ -- refer to \cite[section 2 and Figure 1c]{mafpo4}, tells us that this is impossible for $s$ unigraphic. By the way, it follows that in the case $k=3$ if $s$ is unigraphic then $s$ is simply $14,5^9,3^5$ (i.e. $G$ is the disjoint union of three triangles and five isolated vertices). The same argument excludes the case of exactly one cyclic component and two or more non-trivial trees.

It remains to analyse what happens for exactly one cyclic component $G_1$ and one non-trivial tree $T$. As shown in \cite{mafpo4}, $T=\calC(x_1,\dots,x_\ell)$ is a caterpillar, and we may refer to \cite[Remark 7]{mafpo4} for information about how the elements of $V(T)$ are ordered around the cycle $H$. Let $u_1,\dots,u_i$ be the vertices of $G_1$ in order around $H$. We can choose any among $u_1,u_2$, or $u_2,u_3$, $\dots$, or $u_{i_1},u_1$ to be the two closest vertices on $H$ to the elements of $V(T)$, and moreover we can choose to order $u_1,u_2,\dots,u_{i}$ clockwise or counter-clockwise around $H$. We see that if $u_1u_j$, $3\leq j\leq i-1$ is any edge, then also $u_2u_{j+1}\in E(G)$ by reordering $u_1,\dots,u_i$ around $H$ as above, contradicting planarity -- refer to Figure \ref{pic:001}. Therefore, $G_1=C_i$ is just an $i$-gon.
\begin{figure}[h!]
	\centering
		\includegraphics[width=3.5cm,clip=false]{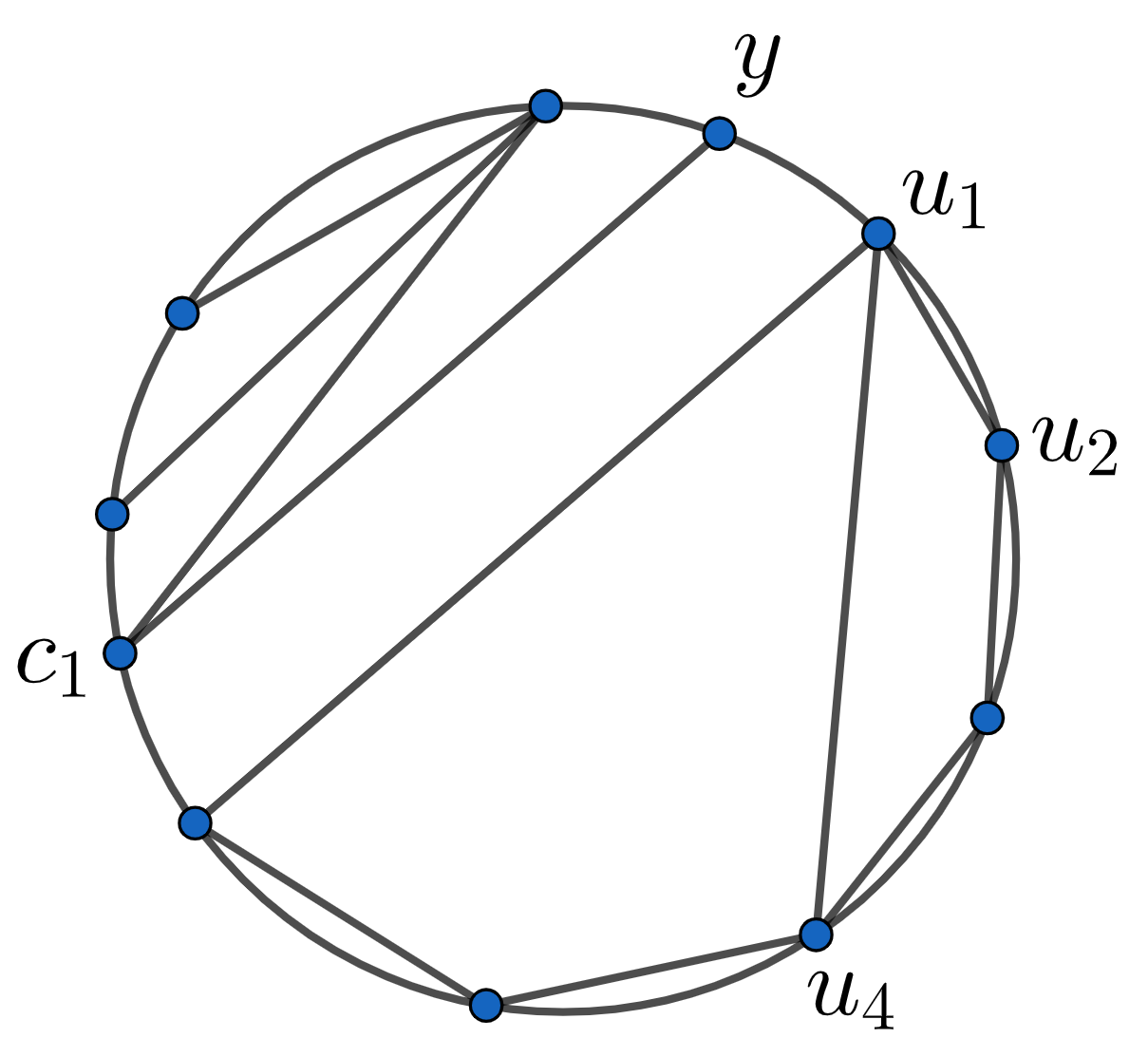}
		\hspace{1cm}
		\includegraphics[width=3.5cm,clip=false]{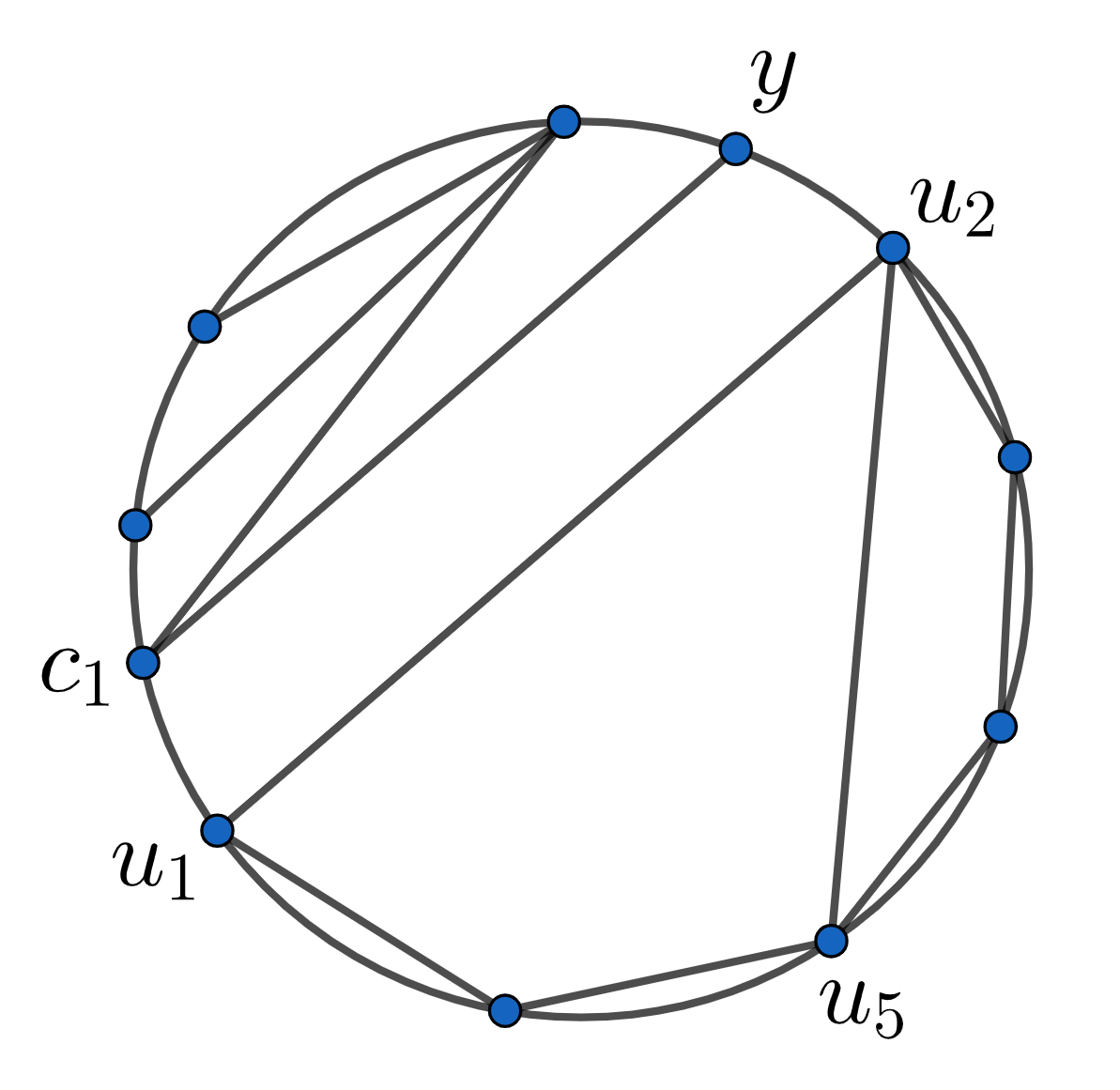}
		\caption{In this example, $u_i,u_1$ are the closest vertices of $G_1$ to $V(T)$ around $H$. Letting $u_1u_4\in E(G)$, by unigraphicity it follows that $u_2u_5\in E(G)$, contradicting planarity.}
	\label{pic:001}
\end{figure}

It follows that
\[s': x_1,\dots,x_\ell,2^i,1^{b},0^a\]
for some $b\geq 2$. Then there is another realisation of $G$ as $\calC(x_1,\dots,x_\ell,2,\dots,2)$, where $2$ appears $i$ times after $x_\ell$ (together with $a$ isolated vertices), contradiction.
\end{proof}

Now assume that $G$ has exactly two cyclic components $G_1,G_2$ (and thanks to Lemma \ref{le:notree} there are no non-trivial tree components). Our goal for the rest of this section is to show that in this case $p\leq 3a-1$. The argument starts similarly to the case of $G_1,T$ of Lemma \ref{le:notree}: to not contradict unigraphicity or planarity, at least one of $G_1,G_2$ is just an $i$-gon, say $G_1=C_i$. 
Thereby, $G_2$ cannot contain acyclic blocks: by contradiction, let $w,w'\in V(G_2)$ be the endpoints of an acyclic block $ww'$. Then we contradict unigraphicity by writing
\[G-ww'-u_1u_i+wu_1+u_iw',\]
where $u_1,\dots,u_i$ are the vertices of $G_1=C_i$ in order around the cycle.

Thanks to \cite[Lemma 5]{mafpo4},
\[
a\geq i+\sum_{j\geq 3}(j-2)\cdot B_G(j),
\]
where $B_G(j)$ counts blocks of $G_2$ bounded by a $j$-gon. As all blocks of $G_2$ are cyclic, we may rewrite
\[a\geq |V(G_1)|+|V(G_2)|-\#\{\text{blocks of }G_2\}\geq |V(G_1)|+\frac{|V(G_2)|-1}{2},\]
so that we have the bound
\[p=1+a+|V(G_1)|+|V(G_2)|\leq 3a+2-|V(G_1)|\leq 3a-1,\]
as claimed.

The arguments of this section imply the following.

\begin{prop}
	\label{prop:rec1}
If $s$ is unigraphic, $a\geq 3$, and $p\geq 3a$, then either $s$ is 
$14,5^9,3^5$, or $G$ (Definition \ref{def:G}) has exactly one non-trivial connected component, and this component contains a cycle.
\end{prop}

\section{Second part of the proof}
\label{sec:g42}
Thanks to Proposition \ref{prop:rec1}, to prove Theorem \ref{thm:1} it remains to inspect the scenario when $G$ (of Definition \ref{def:G}) has exactly one cyclic connected component. Let
\begin{align}
\label{eqn:B}
\notag Z&:=\{\text{vertices of degree 0 in }G\}=\{z_1,\dots, z_a\},\\
\notag Y&:=\{\text{vertices of degree 1 in }G\},\\
B&:=G-Z-Y.
\end{align}

\subsection{$B$ is $2$-connected}
\label{sec:2con}
In this section we suppose that $B$ in \eqref{eqn:B} is $2$-connected. Say that the vertices $b_1,\dots,b_i$ of $B$ are ordered clockwise around the Hamiltonian cycle $H$. If $y\in Y$, $yb_j\in E(G)$ for some $1\leq j\leq i$, then by planarity $y$ lies on either the $b_{j-1}b_j$- or $b_jb_{j+1}$-path in $H$ not containing any of the other vertices of $B$. These two possibilities mean that there is more than one realisation of $s$, unless possibly when either $|Y|\leq 1$, or when every $b_j$ is adjacent to at least one element of $Y$ (cf. \cite[Figure 3a]{mafpo4}).

In the former case, we write $p=1+a+|Y|+i$. Now, $b_ib_1$ and $b_jb_{j+1}$ for all $1\leq j\leq i-1$ are edges of $B$, so that they cannot be edges of $H$. We deduce that there is at least one element of $Z$ between every pair of consecutive vertices of $B$ along $H$, thus $i\leq a$. Therefore, we have the admissible bound on the order of the graph
\[p\leq 1+a+|Y|+a\leq 2a+2.\]

We are left with the case of every $b_j$ adjacent to at least one element of $Y$.
\begin{itemize}
\item
If $B$ is just a cycle of length $i\geq 3$, then $s'$ reads
\[\deg(b_1),\dots,\deg(b_i),1^b,0^d, \qquad\qquad b=\sum_{j=1}^{i}\deg(b_j)-2i.\]
For $i\geq 4$ we may alter $G$ as follows. Let $y\in Y$ be adjacent to $b_1$. We take
\[G-yb_1-b_2b_3+yb_2+b_1b_3.\]
Then $s$ is not unigraphic for $i\geq 4$. On the other hand, when $i=3$, one can check \cite[section 3 and Figure 5b]{mafpo4} that $s$ is unigraphic if and only if $a=3$ 
and at least two of  $\deg(b_1),\deg(b_2),\deg(b_3)$ are equal -- $s$ is of type B1.
\item
Now let $B$ be a triangulated $i$-gon, $i\geq 4$. For $i\geq 6$, we may alter $G$ as follows. Take any triangulated hexagon in the triangulation of $B$, delete one diagonal $b_{j_1}b_{j_2}$ of the hexagon and add another diagonal $b_{j_3}b_{j_4}$, with $j_1,j_2,j_3,j_4$ distinct. This has the effect of decreasing by $1$ the values $\deg(b_{j_1}),\deg(b_{j_2})$ and increasing by $1$ the values $\deg(b_{j_3}),\deg(b_{j_4})$. We then take two vertices in $Y$ adjacent one each to $b_{j_3},b_{j_4}$, and make them adjacent (one each) to $b_{j_1},b_{j_2}$ instead (this may be done without altering the value of $a$). Then $s$ is not unigraphic for $i\geq 6$. Now let $i=4,5$ (and here there is only one way to triangulate a quadrilateral or pentagon). It is straightforward to check that $s$ is unigraphic if and only if all vertices of $B$ have the same degree in $G$, and moreover $a=i$. We get the types B2 and B3 for $i=4,5$ respectively.

\item
It remains to inspect the case where $B$ is neither a cycle nor a triangulated polygon. Then there exist two adjacent regions $R_1$ and $R_2$ in $B$, of respective boundary lengths $i_1,i_2$, such that $i_1\geq i_2$ and $i_1\geq 4$. Similarly to the case where $B$ is a cycle, we delete the edge $b_{j_1}b_{j_2}$ between $R_1,R_2$, add the edge $b_{j_1}b_{j_3}$, where $b_{j_3}\neq b_{j_1}$ is adjacent to $b_{j_2}$ on the boundary of $R_1$, then take a vertex in $Y$ adjacent to $b_{j_3}$, and make it adjacent to $b_{j_2}$ instead (again we do not alter $a$). This is another realisation of $s$, and if $i_1\neq i_2+1$ it is clearly not isomorphic to the initial one, as $i_1$ has decreased by $1$ and $i_2$ increased by $1$. If $i_1=i_2+1$ and $i_2\geq 4$, we perform the transformation above but exchanging the roles of $R_1,R_2$ to reach the same conclusion.

Finally, if $i_1=4$ and $i_2=3$, then $R_1,R_2$ form a pentagon $b_{j_1},b_{j_2},b_{j_3},b_{j_4},b_{j_5}$ with a diagonal $b_{j_1},b_{j_4}$, say. We perform the transformation
\[G-b_{j_3}b_{j_4}+b_{j_1}b_{j_3}-yb_{j_1}+yb_{j_4}\] 
with $y\in Y$, to conclude that $s$ is not unigraphic in this case.
\end{itemize}

We summarise the results of this section as follows.
\begin{prop}
	\label{prop:rec2}
Assume that $s$ is unigraphic, $G$ (Definition \ref{def:G}) has exactly one cyclic component, and $B$ in \eqref{eqn:B} is $2$-connected. Then either $s$ is of type B1, B2, or B3, or $p\leq 2a+2$.
\end{prop}

\subsection{$B$ is not $2$-connected, case 1}
\label{sec:c2}
In this section we suppose that $B$ in \eqref{eqn:B} is not $2$-connected. For $s$ unigraphic, in a cyclic block of $G$ either one, or all vertices are separating in $G$: the idea is similar to the first argument in section \ref{sec:2con}.  

\textit{Assume in this section that each cyclic block of $G$ contains exactly one vertex that is separating in $G$}. We deduce that, if we delete from $G$ all cyclic blocks, we are left with one tree, that is a caterpillar $T=\calC(x_1,\dots,x_\ell)$ due to previous arguments (and may be trivial). It follows that there are two possibilities. Either $T$ is trivial, and then $G$ has exactly one separating vertex $u$, contained in every non-trivial block of $G$; or $T$ is non-trivial, $G$ has two separating vertices $u,v$ (other than the non-degree one vertices $c_j$ of $T$), these $u,v$ have degree $1$ in $T$, and each cyclic block of $G$ contains exactly one of $u,v$. Here we denote by $\calB$ the non-empty set of cyclic blocks containing $u$, and by $\calB'$ the possibly empty set of cyclic blocks containing $v$. Further, we quickly see that $u$ cannot be adjacent to any element of $Y$, and as for $v$, it can only be adjacent to elements of $Y$ in case $\calB'=\emptyset$ (cf. \cite[section 3, and Figures 3a, 4b]{mafpo4}). 

Our next claim is that actually $T=\calC(2,\dots,2)$ (supposing that $T$ is non-trivial). By contradiction, assume that $x_j\geq 3$ for some $1\leq j\leq\ell$. Now by reordering the corresponding vertices $c_j$ on the central path of the caterpillar, we may take $j=1$. Let $B_{1}$ be a cyclic block of $G$, 
$V(B_1)=\{u_1,\dots,u_{i-1},u=u_i\}$,
so that we have in order around $H$ 
\[c_1,u_1,\dots,u_{i-1},\{\text{remaining vertices of }\calB\},u,y,A,\]
with $c_1y\in E(G)$, $\deg_{G}(y)=1$, and $A$ a non-empty set, since $\deg(c_1)\geq 3$ (e.g. Figure \ref{pic:002}, left). We take
\[G-yc_1+c_1u_1-u_1u_2+u_2y\]
moving $y$ to the $u_2u_3$-path in $H$, and moving the isolated vertices of $G$ lying between $u_1,u_2$ to the $c_1u_1$-path in the new graph (Figure \ref{pic:002}, right). This new graph is not isomorphic to $G$, as removing all cyclic blocks now leaves two non-trivial trees (it is essential that $A$ is non-empty, i.e. $\deg(c_1)\geq 3$). Hence $s$ is not unigraphic.
\begin{figure}[h!]
	\centering
	\includegraphics[width=3.5cm,clip=false]{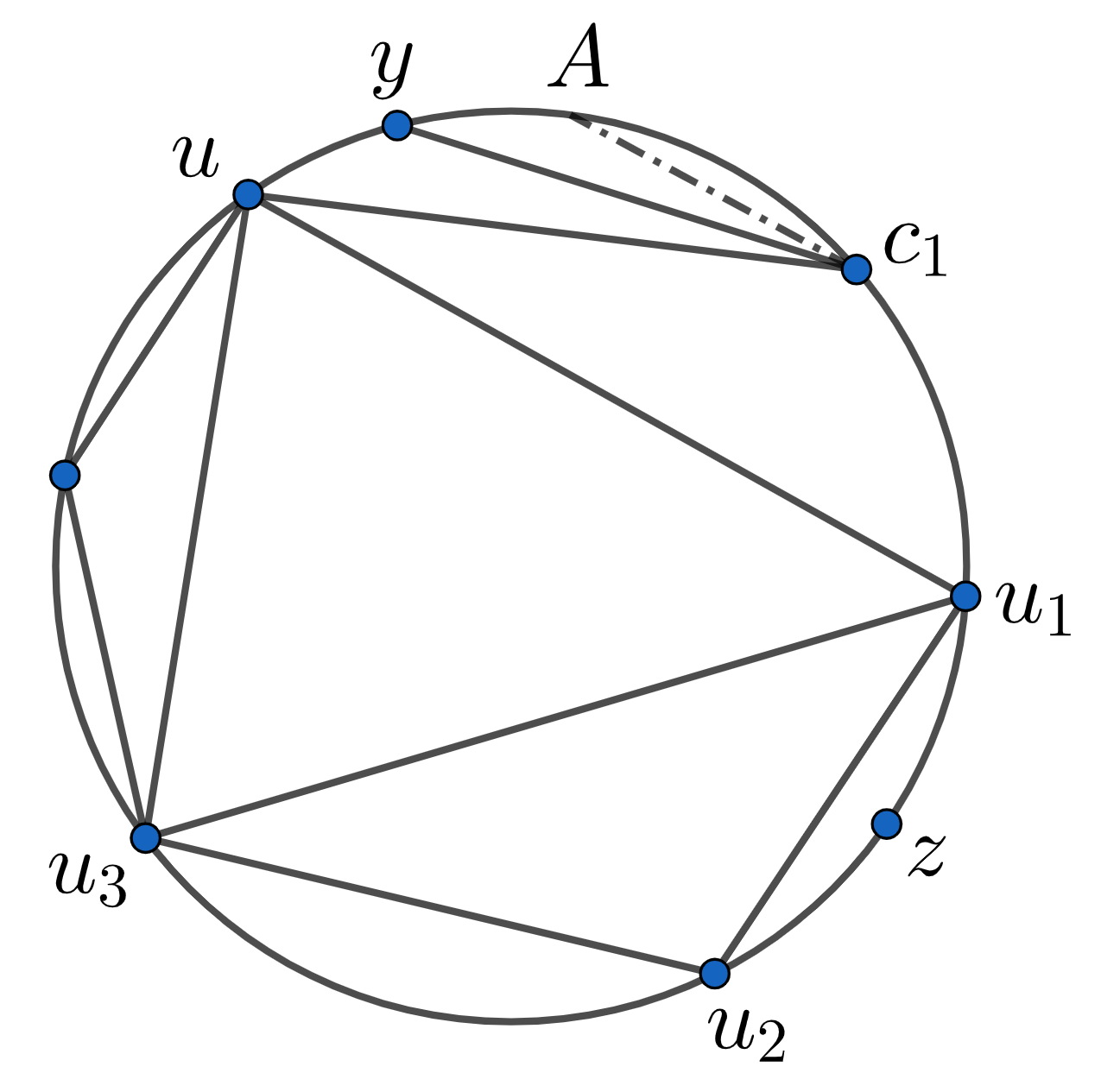}
	\hspace{1cm}
		\includegraphics[width=3.5cm,clip=false]{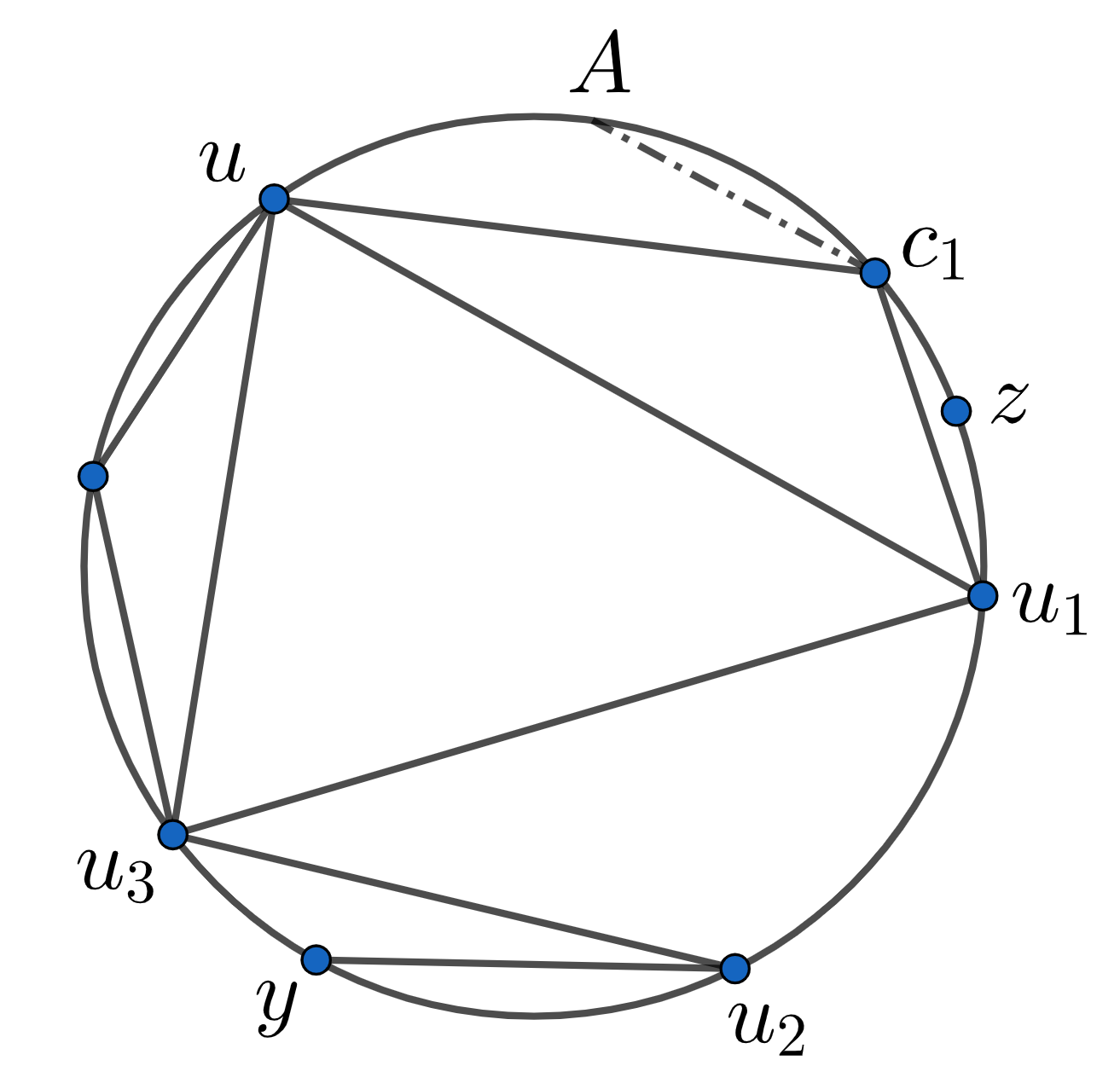}
	\caption{In this example, $z\in Z$, and the dashed-dotted line represents edges incident to a non-empty graph $A$. We transform the graph $G$ on the left to $G-yc_1+c_1u_1-u_1u_2+u_2y$   on the right.}
	\label{pic:002}
\end{figure}


Next, we will show that if a cyclic block $B_1$ of $G$ is adjacent to a final vertex of a path $T$ of length at least $3$, then $B_1$ is a triangle. Consistent with previous notation, call $u$ the vertex of the block that is the starting point on the path, and the subsequent ones along the path $c_1,c_2,\dots,c_\ell,c_{\ell+1}=v$, $\ell\geq 2$. W.l.o.g., the closest vertex of an element of $\calB$ to $u$ along $H$ lies on $B_1$. We perform the transformation
\[G-c_2c_3+c_2u+c_3w-wu,\]
where $w\in V(B_1)$ is the vertex of $B_1$ closest to $u$ along $H$. We then reorder the vertices of $B_1$ and $T$ around $H$ (e.g. Figure \ref{pic:003}) so that 
the minimum value of $a$ for $G$ has not increased. This is possible as deleting $wu$ removes a region, and we added the triangular region of boundary $uc_1c_2$. We may move the isolated vertices of $G$ lying between $u$ and $w$ to the $c_1c_2$-path in the transformed graph. This transformed graph is not isomorphic to $G$ unless $B_1$ is a triangle, as claimed.

\begin{figure}[h!]
	\centering
	\includegraphics[width=3.5cm,clip=false]{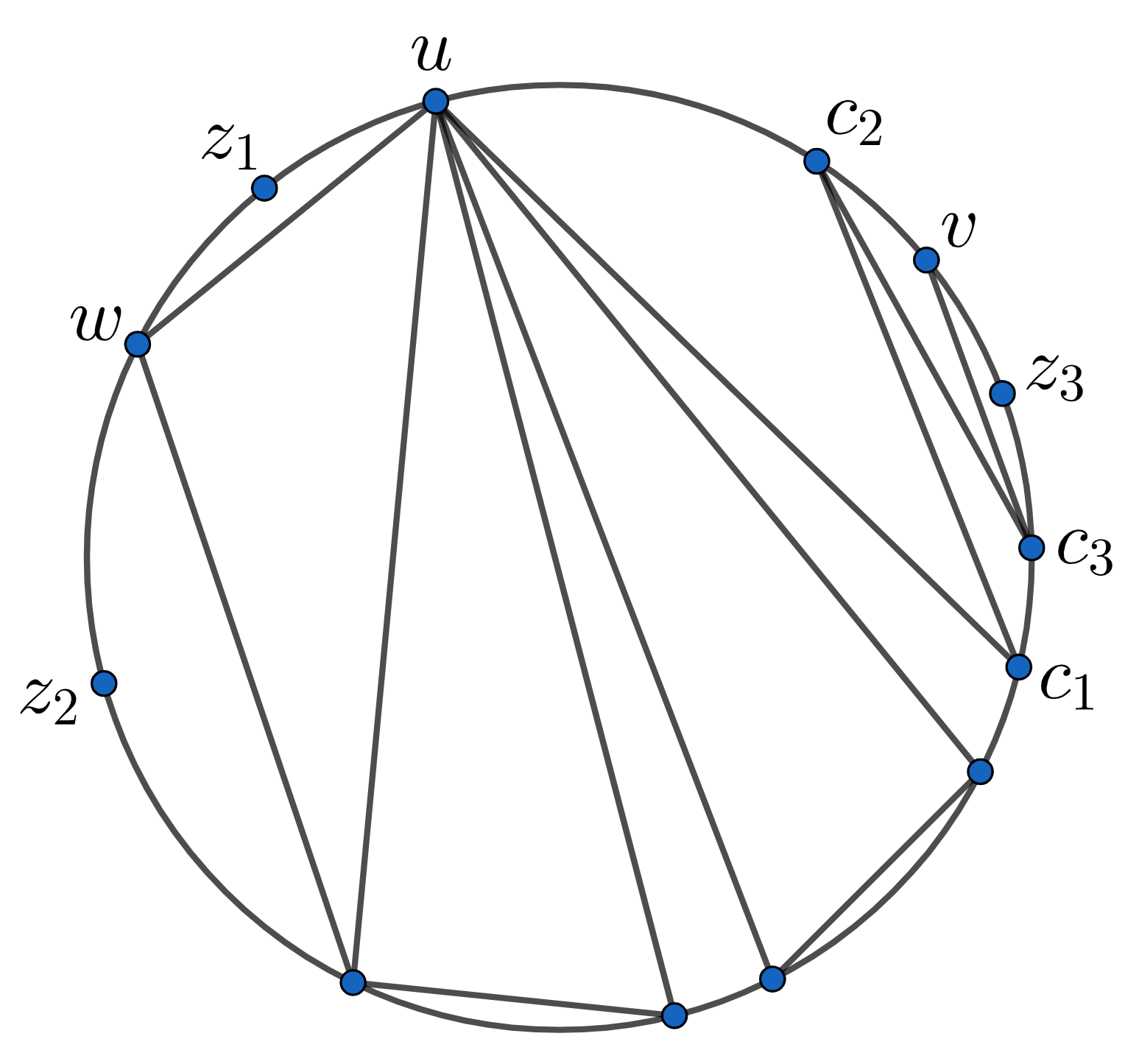}	
	\hspace{1cm}
	\includegraphics[width=3.5cm,clip=false]{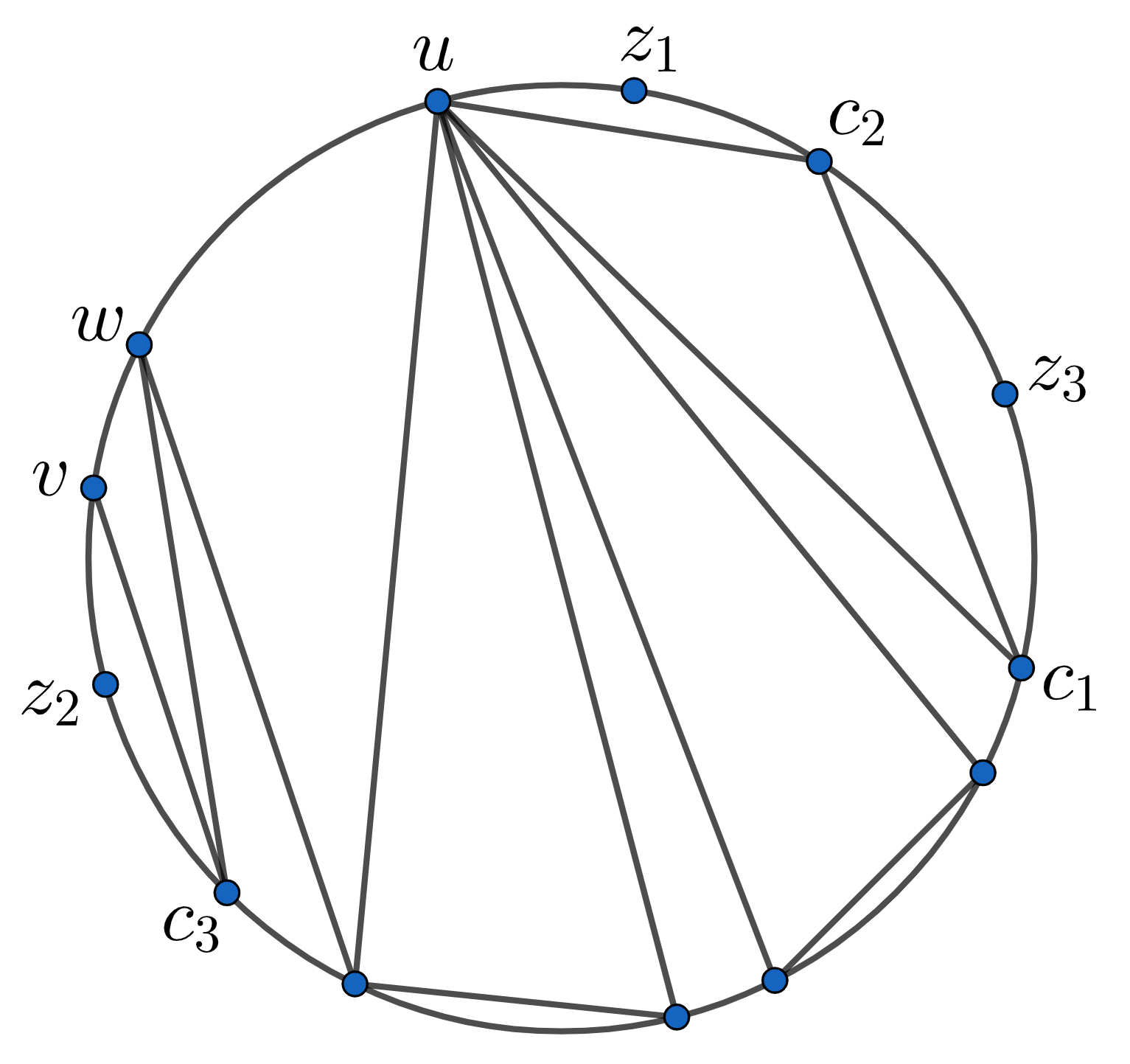}
	\caption{In the depicted example, $z_1,z_2,z_3\in Z$, and $w$ is the vertex of $B_1$ closest to $u$ along $H$. We transform the graph on the left to the one on the right, contradicting unigraphicity since $B_1$ is not a triangle.}
	\label{pic:003}
\end{figure}

Summarising, either $G$ is formed by a set of triangles $\calB$ sharing a vertex $u$, a set of triangles $\calB'$ sharing a vertex $v$, and a $uv$-path, or $G$ is formed by a set of cyclic blocks $\calB$ sharing a vertex $u$, a set of blocks $\calB'$ (either all cyclic or all acyclic) sharing a vertex $v$, and a $uv$-path of length $\leq 2$. In the former scenario, $s'$ is
\[x,y,2^{(x-1)+(y-1)+\ell},0^{2+(x-1)/2+(y-1)/2}\]
i.e. $s$ is of type C.

In the rest of this section, we will focus on the latter scenario. Here we start by ruling out that $v$ is adjacent to two or more elements from $Y$ (i.e., each block of $\calB'$ is cyclic). Indeed, assume for contradiction that $vy_1,vy_2,\dots,vy_l\in E$, $l\geq 2$, thus $\deg_G(v)=l+1$. We remark that planar, cyclic blocks with a region containing all of the vertices (i.e. polygons possibly with some diagonals) have at least two vertices of degree $2$. This is certainly true for a fixed $B_1\in\calB$, and let $w\in V(B_1)$ be such a vertex. Its degree in $G$ is $2$. We then take
\[G-vy_2-\dots-vy_l+wy_2+\dots+wy_l,\]
so that in the new graph the degrees of $v,w$ have swapped, i.e. $s$ is the same but the old and new graphs are non-isomorphic, contradiction.

We can now give an upper bound for the order of $F$. The $3$-polytope contains the vertex of eccentricity one $v_1$, $a$ many of degree $0$ in $G$, at most three on the $uv$-path, and $\#V(B_j)-1$ more for each $B_j\in\calB,\calB'$:
\begin{equation*}
p\leq 1+a+3+\sum_{j=1}^{k}(\#V(B_j)-1),
\end{equation*}
where $k=\#\calB+\#\calB'$. We 
invoke \cite[Lemma 5]{mafpo4},
\begin{equation}
\label{eqn:d2}
a\geq 2+\sum_{j=1}^{k}(\#V(B_j)-2),
\end{equation}
to obtain
\begin{equation}
\label{eqn:pub1}
p\leq 2a+2+k.
\end{equation}
On the other hand, each cyclic block of $G$ is of order at least three, and at least one block is of order at least four, otherwise we would be in the case where all cyclic blocks are triangles. Therefore, \eqref{eqn:d2} also yields
\begin{equation}
\label{eqn:pub2}
a\geq 2+(3-2)(k-1)+(4-2)=3+k.
\end{equation}
We substitute \eqref{eqn:pub2} into \eqref{eqn:pub1} to see that
$p\leq 3a-1$ in this scenario. 

The arguments of this section imply the following.

\begin{prop}
	\label{prop:rec3}
	Assume that $s$ is unigraphic, $G$ (Definition \ref{def:G}) has exactly one cyclic component, and $B$ in \eqref{eqn:B} is not $2$-connected. Suppose further that there is no cyclic block of $G$ that contains only vertices that are separating in $G$. Then either $G$ is of type C, or $p\leq 3a-1$.
\end{prop}

\subsection{$B$ is not $2$-connected, case 2}
\label{sec:c3}
In this section we suppose that $B$ in \eqref{eqn:B} is not $2$-connected, and moreover that there exists a cyclic block $B_1$ of $G$ that contains only vertices that are separating in $G$.

We claim that there exists a cyclic endblock $B_j\neq B_1$. Indeed, there are always at least two endblocks, and if both were acyclic, then there would be in $G$ two disjoint paths of three or more vertices each, contradicting previous arguments (this is because one endpoint of an acyclic endblock of $B$, i.e. copy of $K_2$, is necessarily adjacent to one or more elements of $Y$, or the endpoint would be itself an element of $Y$, and thus the copy of $K_2$ would not be a block of $B$ in the first place). Hence $B_j$ exists.

Planar, cyclic blocks with a region containing all of the vertices (i.e. polygons possibly with some diagonals) have at least two vertices of degree $2$. Let $w$ be a vertex of $B_j$ of degree $2$ in $B_j$, non-separating in $B$, and $u$ a vertex of degree $2$ in $B_1$ such that in $G-u$ there is still a path between $B_1,B_j$ ($u$ always exists, as $B_1$ has at least two vertices of degree $2$ in $B_1$) -- refer to Figure \ref{pic:004}. We may transform $G$ by moving the adjacencies of $u$ not in $V(B_1)$ to $w$ instead, and vice versa the adjacencies of $w$ not in $V(B_j)$ to $u$ (this does not affect $s$). By unigraphicity, we conclude that this operation produces an isomorphic graph. Now $u$ is separating in $G$ by definition, and by construction any neighbour of $w$ (save for the two in $V(B_j)$) belongs to $Y$ of \eqref{eqn:B} (it has degree $1$ in $G$). It follows that $u,w$ are adjacent to the same number $\alpha-2\geq 1$ of vertices in $Y$, where $\deg_G(u)=\deg_G(w)=\alpha$. 
\begin{figure}[h!]
	\centering
	\includegraphics[width=4cm,clip=false]{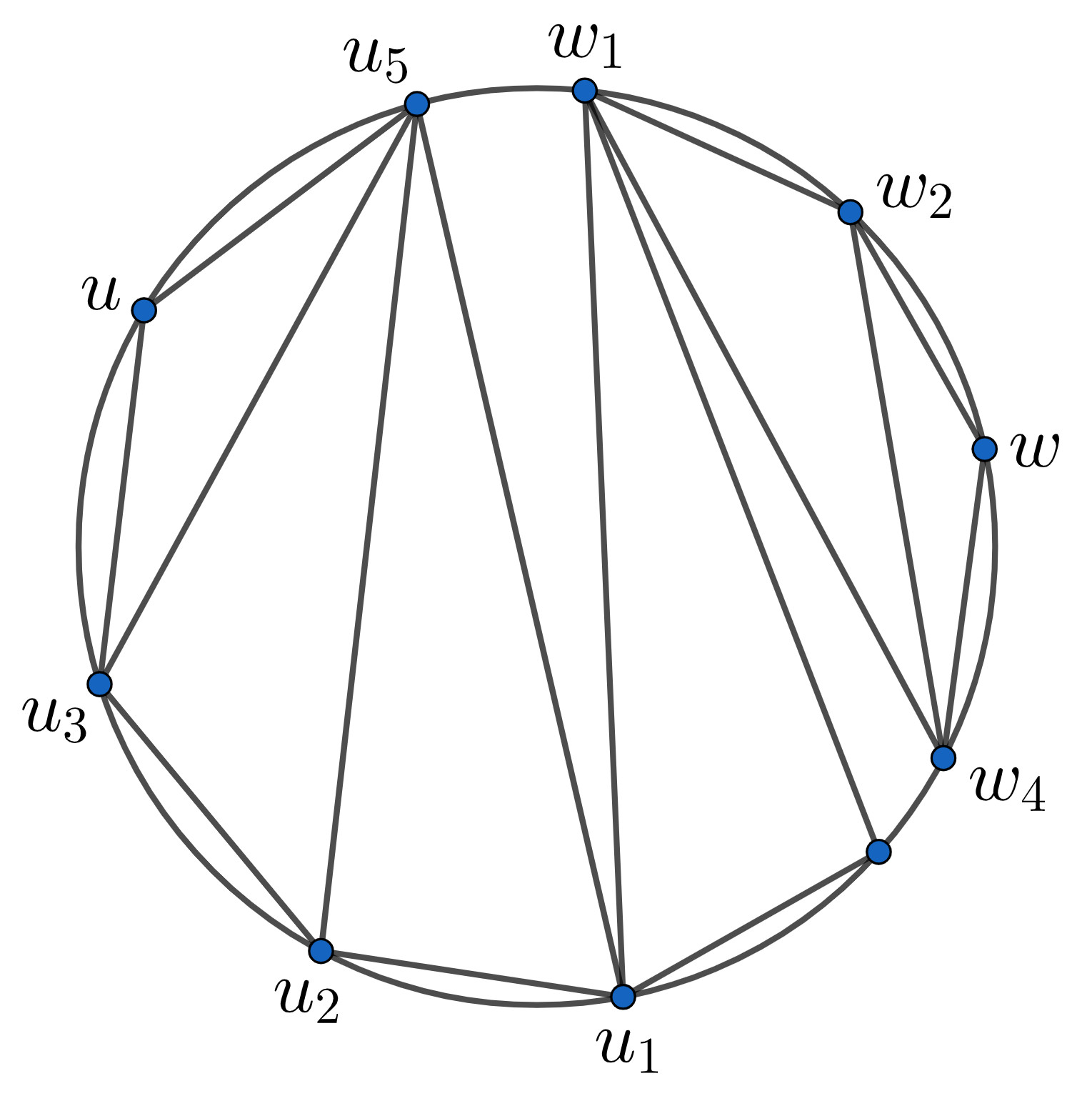}
	\caption{In this example, we assume that $B_j$ is an endblock, $V(B_j)=\{w_1,w_2,w,w_4\}$, and $V(B_1)=\{u_1,u_2,u_3,u,u_5\}$. Moreover, all vertices of $B_1$ are separating in $G$, $w_1$ is separating in $B$, and $w_2,w,w_4$ are separating in $G$ but not in $B$ (only a subgraph of $G$ is depicted). We note that $w$ is a vertex of $B_j$ of degree $2$ in $B_j$, that is non-separating in $B$. Also, $u$ is a vertex of degree $2$ in $B_1$ s.t. in $G-u$ there is still a path between $B_1,B_j$. Swapping adjacencies between $w,u$ (other than $ww_2,ww_4,uu_3,uu_5$) must produce an isomorphic graph. Therefore, all other vertices adjacent to $u$ must be in $Y$, and $\deg_G(u)=\deg_G(w)=\alpha\geq 3$.}
	\label{pic:004}
\end{figure}

The arguments of section \ref{sec:2con} now imply that \textit{all} vertices of a cyclic endblock $B_j$ that are non-separating in $B$ are adjacent to one or more elements of $Y$. 
Note that all vertices of $B_j$ are separating in $G$, and exactly one, $w_0$ say, is separating in $B$.

Via an argument similar to previous sections, we now show that actually an endblock of $G$ cannot be a copy of $K_2$, except possibly if there are only the two blocks $B_1$ and $K_2$ in $G$. Indeed, by contradiction call $V(K_2)=\{w_0',w_1'\}$, $w_0'$ separating in $B$ and $w_1'$ non-separating in $B$. By construction, $w_1'$ is adjacent to $y_1',\dots,y_i'\in Y$, $i\geq 1$, and we have seen above that there exists $y\in Y$, $w_1'y\not\in E$. We perform
\begin{equation}
\label{eqn:op2}
G-w_1'y_1'-\dots-w_1'y_i'+yy_1'+\dots+yy_i'
\end{equation}
and obtain a new graph, that is non-isomorphic to $G$ as soon as there are two or more cyclic blocks in $G$. We reach a contradiction unless there are only two blocks $B_1$ and $K_2$, and moreover $B_1$ must be a cycle in this case.

Still by the arguments of section \ref{sec:2con}, a cyclic endblock $B_j$ is either a triangle or a triangulated quadrilateral or pentagon: the arguments for the only block $B$ of $G$ in section \ref{sec:2con} apply here to $B_j$, since all vertices of $B_j$ are separating in $G$, and exactly one is separating in $B$.

If the cyclic endblocks of $G$ are all triangles, then we may possibly have $\alpha=3$; if one of them is a triangulated quadrilateral or pentagon, $\alpha\geq 4$. Let's see that actually there cannot be two endblocks $B_j,B_{j'}$ that are both triangles. Indeed, in this scenario, $V(B_j)=\{w_0,w_1,w_2\}$, $w_0$ separating in $B$, $w_1y_1,w_2y_2\in E$, $y_1,y_2\in Y$, and likewise $V(B_{j'})=\{w_0',w_1',w_2' \}$, $w_0'$ separating in $B$, $w_1'y_1',w_2'y_2'\in E$, $y_1',y_2'\in Y$. We consider the transformation
\[G-w_1y_1-w_2y_2+w_1w_2'+w_2w_2'-w_0'w_2'-w_1'w_2'+w_1'y_1+w_0'y_2\]
that alters $F$ but not $s$ to 
reach a contradiction and rule out this scenario.

To summarise, either there are exactly two blocks in $G$ -- a cycle and a copy of $K_2$, or at least one endblock is a triangulated quadrilateral or pentagon, and $\alpha\geq 4$. Suppose for the moment that we are in the latter case. Let \[V(B_j)=\{w_0,w_1,\dots,w_l\},\]
with $w_0$ separating in $B$, and $2\leq l\leq 4$. Since $w_1,\dots,w_l$ are adjacent to one or more elements of $Y$, it must hold that $\deg_{B_j}(w_0)=2$ by unigraphicity, with $w_0w_1,w_0w_2\in E$, say -- refer to Figure \ref{pic:08}. We then transform $G$ by
\begin{equation}
\label{eqn:op}
G-uy_1-uy_2+uw_1+uw_2-w_0w_1-w_0w_2+w_0y_1+w_0y_2,
\end{equation}
where $u$ is adjacent to at least $2$ elements $y_1,\dots,y_{\alpha-2}\in Y$, and in $G-u$ there is still a path between $B_j$ and a vertex in the same block as $u$ -- refer to Figure \ref{pic:10}.

\begin{figure}[h!]
	\centering
	\begin{subfigure}{.54\textwidth}
		\centering
		\includegraphics[width=3.5cm,clip=false]{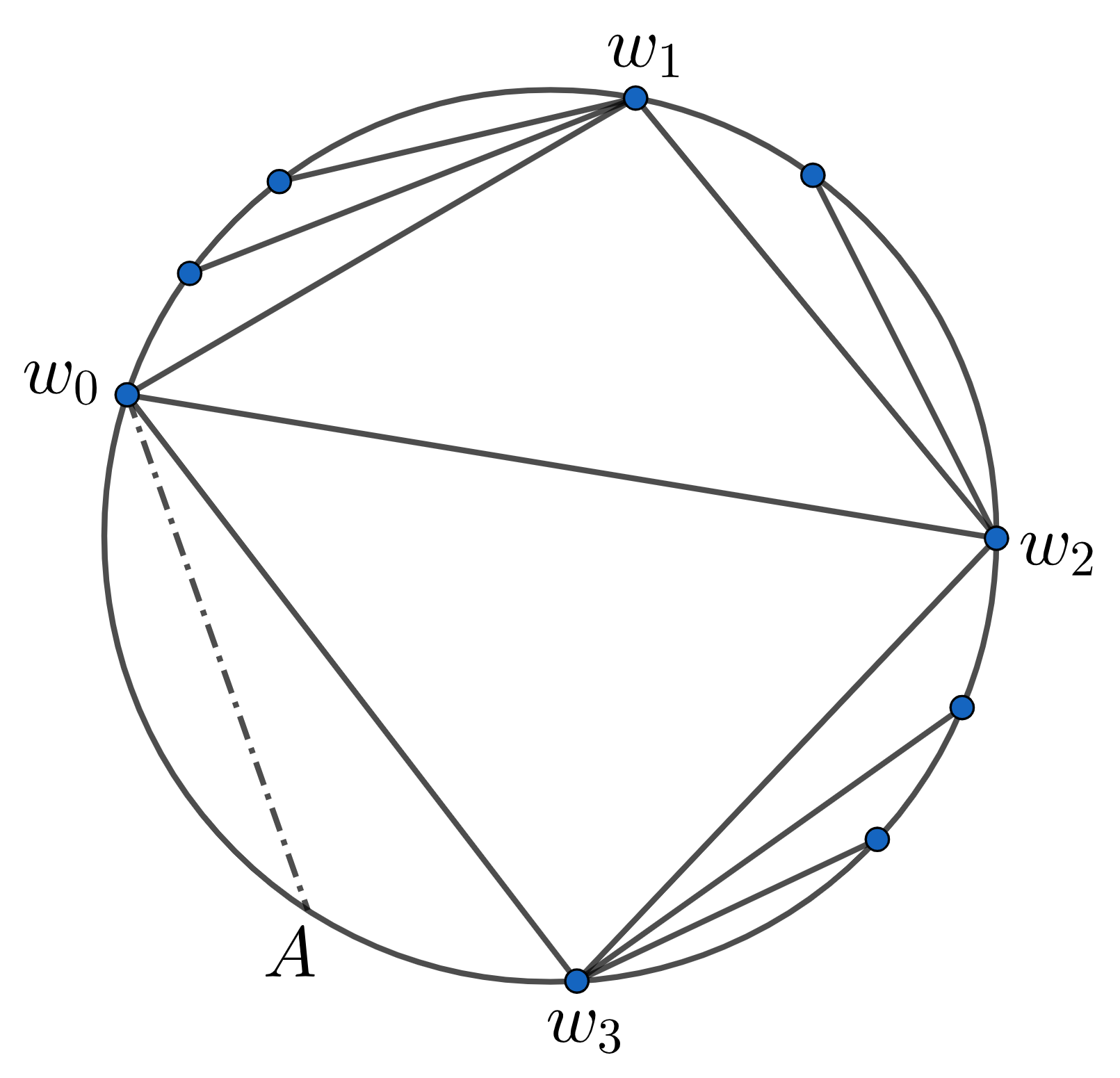}	
		\hspace{1cm}
		\includegraphics[width=3.5cm,clip=false]{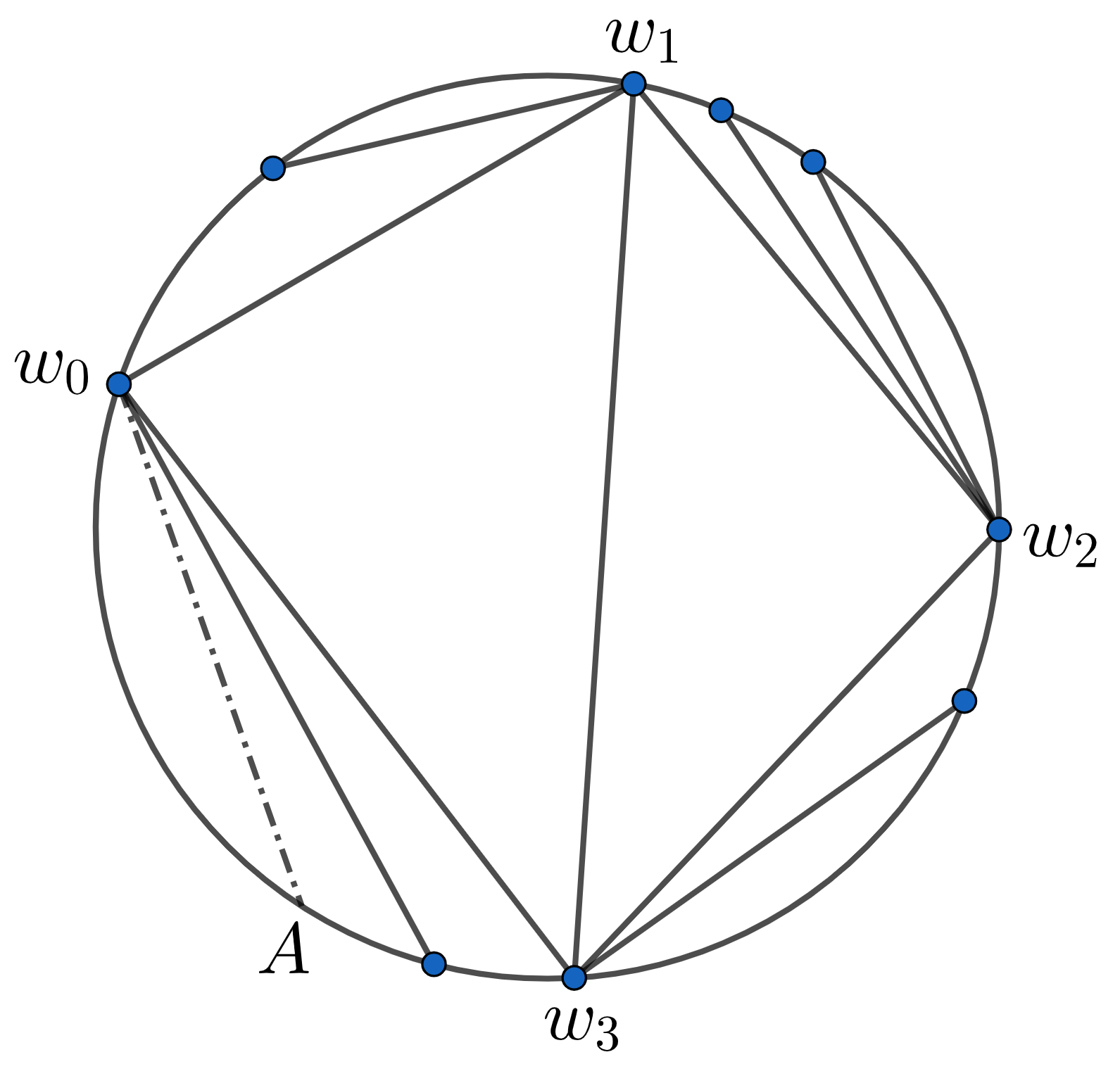}
		\caption{The endblock $B_j$ is a diamond graph. The dashed-dotted line represents edges incident to a non-empty graph $A$. By unigraphicity, the degrees of $w_1,w_2,w_3$ in $G$ must be the same value $\alpha\geq 4$ (in this example $\alpha=4$.) The transformation from the first graph to the second graph is applicable if and only if $\deg_{B_j}(w_0)\geq 3$.}
		\label{pic:08}
	\end{subfigure}
	\hspace{0.25cm}
	\begin{subfigure}{.42\textwidth}
		\centering
		\includegraphics[width=3.5cm,clip=false]{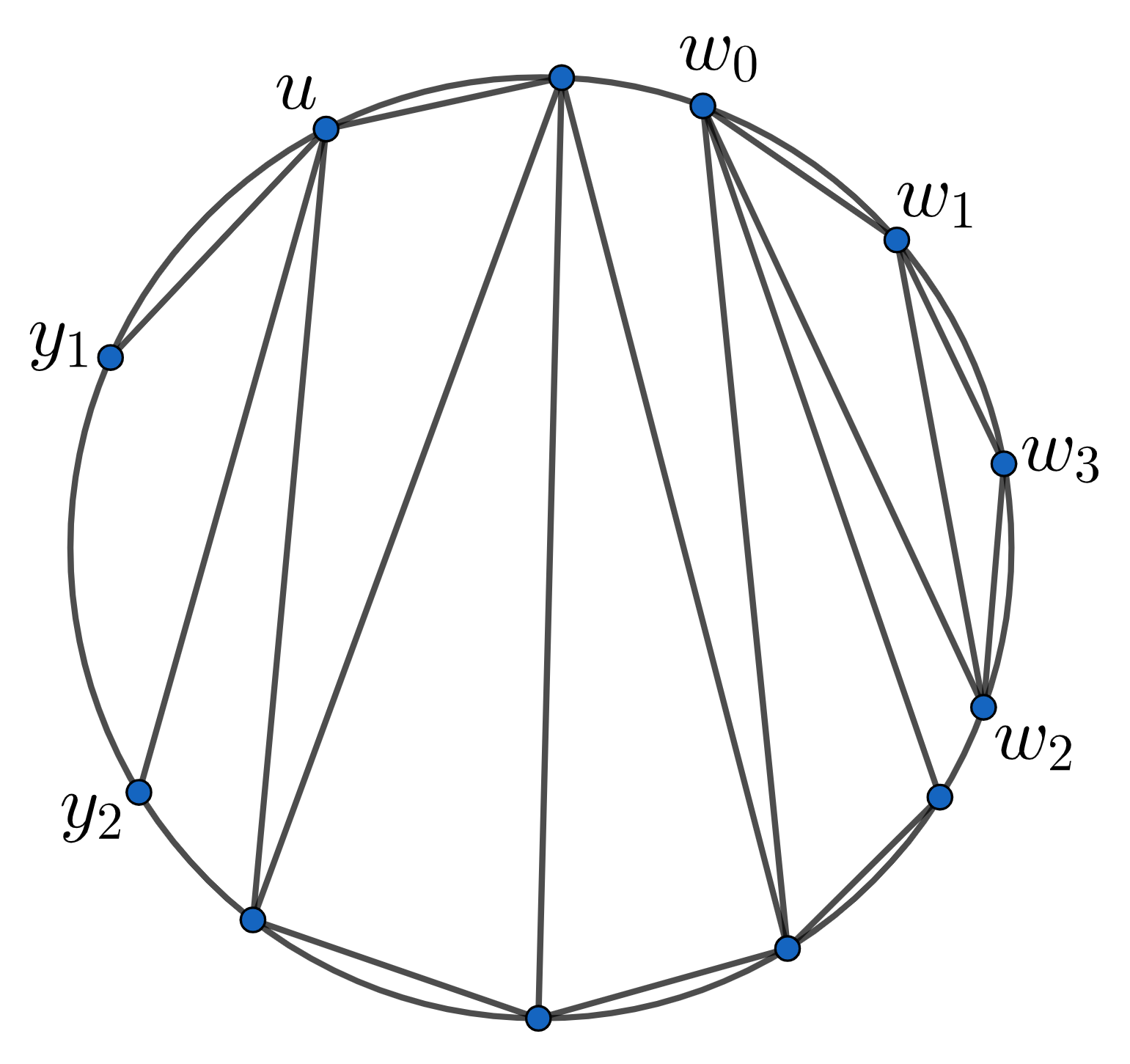}	
		\hspace{1cm}
		\includegraphics[width=3.5cm,clip=false]{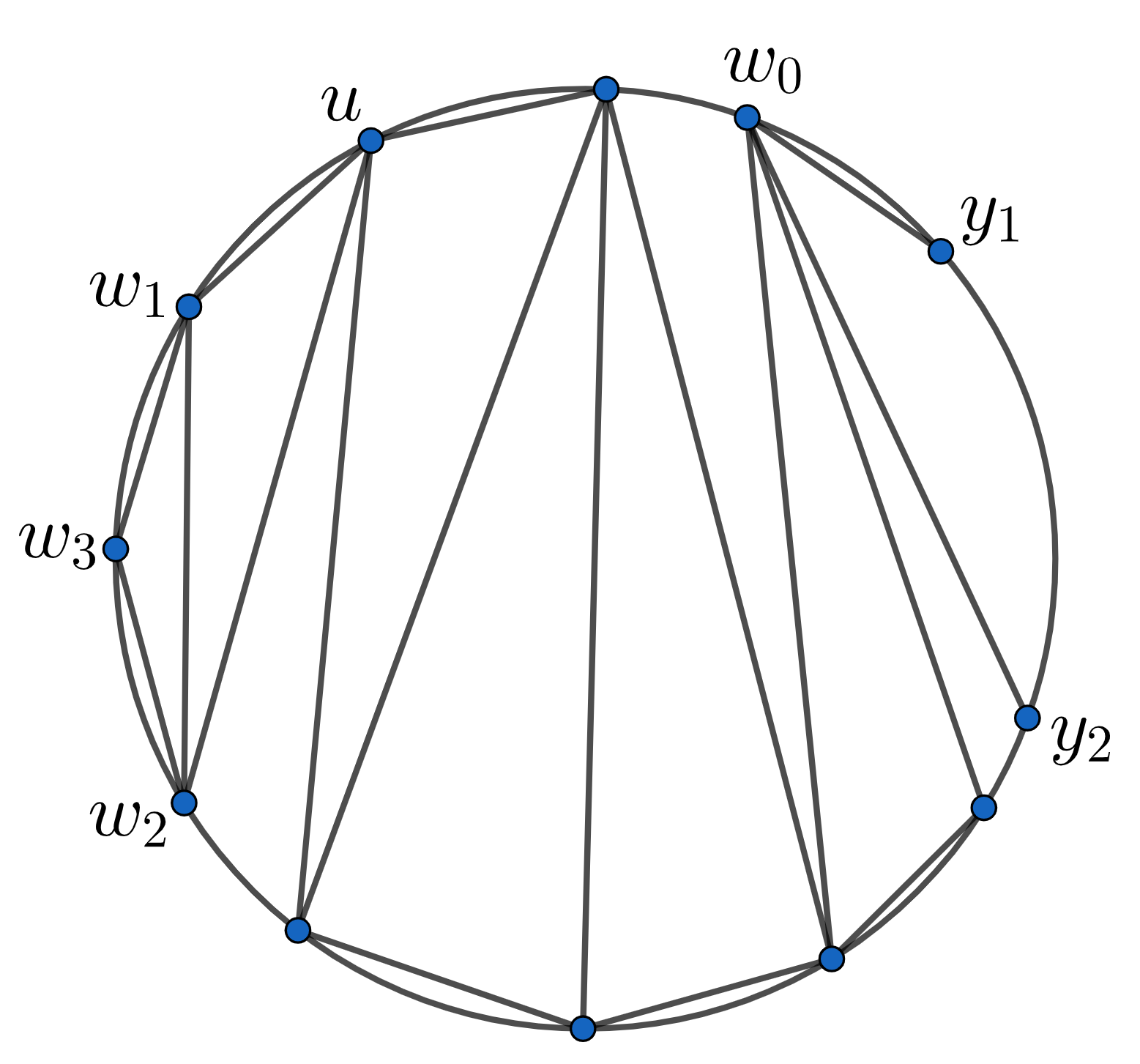}
		\caption{An application of \eqref{eqn:op} transforms the first graph into the second one. Here $\alpha=4$, and the endblock $B_j$ is a diamond graph ($l=3$). Only a subgraph of $G$ is depicted.}
		\label{pic:10}
	\end{subfigure}
	\caption{The case $\alpha\geq 4$.}
	\label{pic:005}
\end{figure}

By unigraphicity, \eqref{eqn:op} must not change $F$, that is to say, all possible $u$ must belong to a block sharing the vertex $w_0$ with $B_j$. There are thus exactly two cyclic blocks in $G$, one of them being a triangulated quadrilateral or pentagon, and the other a triangle or triangulated quadrilateral or pentagon. It is straightforward to see that
\[s': k^b,1^c,0^a,\]
with $6\leq b\leq 10$, $k\geq 5$, $c$ depending only on $k$, and $5\leq a\leq 8$. One checks all possibilities for $b$ to rule out this option entirely.

Therefore, we finally see that $B$ has exactly two blocks, one of them being a copy of $K_2$ on the vertices $w_0,w_1'$ say, and the other (i.e. $B_1$) a cycle. We have seen that $B_1$ must also be a triangle or triangulated quadrilateral or pentagon, hence it is a triangle, and we write
$V(B_1)=\{w_0,w_1,w_2\}$. One quickly sees that for $s$ unigraphic, the degrees in $G$ of $w_0,w_1,w_2,w_1'$ must all be equal. We have obtained a graphic sequence $s$ of type D. We summarise the work of this section as follows.
\begin{prop}
	\label{prop:rec4}
Assume that $s$ is unigraphic, $G$ (Definition \ref{def:G}) has exactly one cyclic component, and $B$ in \eqref{eqn:B} is not $2$-connected. Suppose further that there is a cyclic block of $G$ that contains only vertices that are separating in $G$. Then $s$ is of type D.
\end{prop}

\subsection{Concluding the proof of Theorem \ref{thm:1}}
\label{sec:end}
Gathering the results of Propositions \ref{prop:rec1}, \ref{prop:rec2}, \ref{prop:rec3}, and \ref{prop:rec4}, we deduce that if $s$ is unigraphic as a $3$-polytope of radius one satisfying $a\geq 3$, then either $s$ is one of B1, B2, B3, C, D, or $p\leq 3a-1$, or $s$ is $14,5^9,3^5$. The proof of Theorem \ref{thm:1} is thus complete.

\addcontentsline{toc}{section}{References}
\bibliographystyle{abbrv}
\bibliography{bibgra}

\Addresses

\end{document}